\documentclass{article}
\usepackage{amsmath}
\usepackage{tabulary,graphicx,times,caption,fancyhdr,amsfonts,amssymb,amsbsy,latexsym,amsthm}
\usepackage[colorlinks,linkcolor=black,anchorcolor=black,citecolor=black]{hyperref}
\newtheorem{theorem}{Theorem}
\newtheorem{lemma}{Lemma}

\usepackage[utf8]{inputenc}
\usepackage{url,multirow,morefloats,floatflt,cancel,tfrupee,textcomp,colortbl,xcolor,pifont}
\usepackage[nointegrals]{wasysym}
\makeatletter

\usepackage{ifxetex}
\ifxetex\else
  \usepackage{dblfloatfix}
\fi

\@ifundefined{subparagraph}{
\def\subparagraph{\@startsection{paragraph}{5}{2\parindent}{0ex plus 0.1ex minus 0.1ex}%
{0ex}{\normalfont\small\itshape}}%
}{}

\def\URL#1#2{\@ifundefined{href}{#2}{\href{#1}{#2}}}

\def\UrlOrds{\do\*\do\-\do\~\do\'\do\"\do\-}%
\g@addto@macro{\UrlBreaks}{\UrlOrds}

\makeatother

\usepackage[paperheight=11in,paperwidth=8.3in,margin=2.5cm,headsep=.7cm,top=2.5cm]{geometry}
\usepackage[T1]{fontenc}

\renewenvironment{abstract}
	{\trivlist\item[]\leftskip0pt\par\vskip4pt\noindent
  	\textbf{\abstractname}\mbox{\null}\\}
	{\par\noindent\endtrivlist}

\def\keywords#1{\par\medskip\par\noindent\textbf{Keywords}: #1\par}

 \date{} \emergencystretch 8pt

\makeatletter
\def\author#1{\gdef\@author{\hskip-\tabcolsep%
	\parbox{\textwidth}{\raggedright\bfseries#1\\[1pc]}}}
\def\address[#1]#2{\g@addto@macro\@author{\\\hskip-\tabcolsep\parbox{\textwidth}{\raggedright%
	\normalsize\normalfont\textsuperscript{#1}#2}}}
\let\addresslink\textsuperscript
\def\correspondence#1{\g@addto@macro\@author{\\\hskip-\tabcolsep\parbox{\textwidth}{\raggedright%
	\vspace*{10pt}\normalsize\normalfont~\\#1~\\[12pt]}}}
\def\email#1{\g@addto@macro\@author{\\\hskip-\tabcolsep\parbox{\textwidth}{\raggedright%
	\normalsize\normalfont Emails: #1}}}

\def\title#1{\gdef\@title{\vspace*{-30pt}%
	\raggedright\textbf{\@journaltitle}~\\%
  \raggedright\bfseries\ifx\@articleType\@empty\vspace*{20pt}\else%
  \vspace*{20pt}\@articleType\vspace*{20pt}\\\fi#1}}
\let\@journaltitle\@empty \def\journaltitle#1{\gdef\@journaltitle{{\normalfont\itshape#1}}}
\let\@articleType\@empty \def\articletype#1{\gdef\@articleType{{\normalfont\itshape#1}}}

\let\@runningHead\@empty \def\RunningHead#1{\gdef\@runningHead{{\normalfont #1}}}

\usepackage{fancyhdr}
\fancypagestyle{headings}{\fancyhf{}
  \fancyhead[R]{\itshape\@runningHead}
  \fancyfoot[C]{\thepage}}
\fancypagestyle{plain}{%
	\fancyhf{}\fancyhead[R]{\LaTeX\ template for Hindawi journal articles}
  \fancyfoot[C]{\thepage}}
\makeatother

\usepackage[%
	numbers,sort&compress%
  ]{natbib}

\usepackage{float,xcolor}

\journaltitle{Complexity}
\articletype{Research Article} 

\begin{document}

\title{Backward Stochastic Differential Equations (BSDEs) Using Infinite-dimensional Martingales with Subdifferential Operator}

\author{%
		Pei Zhang\addresslink{1,2},
  	Adriana Irawati Nur Ibrahim\addresslink{1} and
  	Nur Anisah Mohamed\addresslink{1}
    }
		
\address[1]{Institute of Mathematical Sciences, Faculty of Science, Universiti Malaya, 50603 Kuala Lumpur, Malaysia}
\address[2]{School of Mathematics and Statistics, Suzhou University, Suzhou 234000, Anhui, China}

\correspondence{Correspondence should be addressed to 
    	Adriana Irawati Nur Ibrahim: adrianaibrahim@um.edu.my}

\email{s2023288siswa.um.edu.my (Pei Zhang), nuranisah mohamed@um.edu.my (Nur Anisah Mohamed)}%

\RunningHead{Complexity}
\maketitle 

\begin{abstract}
In this paper, we focus on a family of backward stochastic differential equations (BSDEs) with subdifferential operators that are driven by infinite-dimensional martingales. We shall show that the solution to such infinite-dimensional BSDEs exists and is unique. The existence and uniqueness of the solution is established using Yosida approximations. Furthermore, as an application of the main result, we shall show that the backward stochastic partial differential equation driven by infinite-dimensional martingales with a continuous linear operator has a unique solution under the condition that the function $F$ equals zero.
\end{abstract}
    
\section{Introduction}
In 1990, Pardoux and Peng \cite {Par90} firstly proposed the general nonlinear case of backward stochastic differential equations (BSDEs): let $(\xi,f)$ include a square integrable random variable $\xi$ and a progressively measurable process $f$, and let ${W_t}_{(0\leq t\leq T)}$ be a  $k$-dimensional Brownian process. Then the existence and uniqueness of an adapted process $(Y,Z)$ is a solution to the
following type of BSDEs: 
\begin{eqnarray*}
Y_t=\xi+\int_{t}^{T}f(s,Y_s,Z_s)\,{\rm d}s-\int_{t}^{T}Z_s\,{\rm d}W_s, \ \ 0\leq t\leq T.
\end{eqnarray*}
Since then, many scholars have begun to carry out more in-depth researches on BSDEs. As a result, BSDEs have developed rapidly, whether in their own development or in many other related fields such as financial mathematics, stochastic control, biology, financial futures market, theory of partial differential equations, and stochastic games. Reference can be made to Karoui $et\ al.$\cite{Kar97},  Hamadene and  Lepeltial \cite{Ham95}, Peng \cite{Pen97}, Ren and Xia \cite{Ren06} among others.
Among the BSDEs, Pardoux and R\u{a}\c{s}canu \cite{Par98} mulled over backward stochastic differential equations involving a subdifferential operator, which are also dubbed Backward Stochastic Variational Inequalities (BSVIs), and also utilized them with the Feymann-Kac formula to represent a solution of the multivalued parabolic Partial Differential Equations (PDEs).
Pardoux and R\u{a}\c{s}canu \cite{Par99} showed that the result can be easily expanded to a spatial setting Hilbert by giving examples of backward stochastic partial differential equations with solution.
Diomande and Maticiuc \cite{Dio14} used a mixed Euler-Yosida scheme, to prove the existence of the solution of the multivalued BSDE with time-delayed generators; Maticiuc and Rotenstein \cite{Mat12} had provided the numerical results of the multivalued BSDEs. Boufoussi \cite {Bou07} proved a type of generalized backward doubly stochastic differential equation with a symmetric backward stochastic It\^{o} integral has an existence and uniqueness solution.
Wang and Yu \cite{Wang22} explored this problem with anticipated type generalized backward doubly stochastic differential equation.
Instead of normal Brownian motion as interference source, Yang $et.\ al.$ \cite{Yan17} showed the existence and uniqueness of the solution for a type of BSDEs driven by finite G-Brownian process with subdifferential operator by using the Method of Approximation of Moreau-Yosida.

Some authors have also obtained results in the type spaces of ${L^p}$, among which Briand $et\  al.$ \cite{Bri03} obtained a priori estimate and demonstrate the existence and uniqueness of solutions in $L^p, p>1$. Under normal conditions, Fan $et.\ al.$ \cite{Fan16} studied bounded solutions, $L^{p}(p > 1)$ solutions and $L^1$ solutions of one-dimensional equation.

Instead of focusing on one-dimensional BSDEs ($Y \in$ $\mathbb{R}$), it is possible to extend to multi-dimensional settings. Bahlali \cite{Bah02} had proven the existence, uniqueness and stability of the solution for multi-dimensional BSDEs with a local monotonous coefficient. Maticiuc and R\u{a}\c{s}canu \cite{Mat15} extended the existence and uniqueness results of the previous work of Pardoux and R\u{a}\c{s}canu \cite{Par99} by supposing a weaker boundedness condition for the generator and by considering random time interval $[0, T]$, the Lebesgue–Stieltjes integral terms,where fixed convex boundary is induced by the subdifferential of an appropriate lower semicontinuous convex function.
In other words, the infinite dimensional condition, and in the ${L^p (p\geq2)}$ space, a generalized version of the multivalued BSDE taken into account at a random time interval remove bracket. 
 R\u{a}\c{s}canu \cite{Ras18} proved that in the case of $p\geq2$, the variational solution is a strong one since they have certified the uniqueness of that solution.

Likewise, martingale has wider application than the Brownian Motion. The properties of the martingale described may not hold true, and one generally needs to enter more martingale into the response. 
 Hamaguchi \cite{Ham20} proposed an endless dimensional BSDE driven by a barrel shaped martingale, demonstrate the presence and uniqueness of the arrangement of such boundless dimensional BSDEs and show the grouping of arrangements of relating BSDEs. El Karoui and Huang \cite{Kar97} studied BSDEs driven by finite dimensions martingales. Al-Hussein \cite{Al09} demonstrated an aftereffect of presence, uniqueness of solution of a BSDEs which driven by a limitless dimensional martingale, and applied the outcome to track down a special answer for a regressive stochastic fractional differential condition in boundless measurements.
Because the case of $p=2$ is more common and $p>2$ is more complex in $L^p$ space, hence, it is necessary to study BSDEs with sub-differential operator, whose drives are infinite-dimensional martingales in $L^2$ space.

Considering subdifferential operator and martingale simultaneously, Nie \cite{Nie15} concentrated on the existence and uniqueness of the solution to a multi-dimensional forward-backward stochastic differential equation (FBSDE) with subdifferential operator in the backward condition, moreover the backward equation is reflected on the boundary of a closed convex area. But as far as we know, the research on infinite dimensional martingale has not been done before. 

The purpose of this paper is to consider a class of BSDEs driven by infinite dimensional martingales with subdifferential operator  of the following type:
\begin{equation}\label{eq1}
\left\{
\begin{array}{ll}
  \displaystyle \,{\rm d}Y_{t}+F(t,Y_{t},Z_{t}\mathcal{Q}_{t}^{1/2})\,{\rm d}{t}\in\partial\varphi(Y_t)\,{\rm d}{t}+Z_{t}\,{\rm d}{M_{t}}+\,{\rm d}{N_{t}},\ 0\leq t\leq T,\\
\\
 \displaystyle Y(T)=\xi.
\end{array} \right.
  \end{equation}

Eq. (\ref{eq1}) is written in the context of a completion probability space $(\Omega,\mathcal{F},P)$ with a continuous filter \{$\mathcal{F}_{t}$\}$_{t \geq 0}$ on the right side. $\xi$ is a random variable, given as a final value, the function $F$ is a mapping from $\Omega\times[0,\infty)\times H\times L^{2}(H)$ to $H$, $M$ is a continuous martingale in the space of $H$, besides ${\mathcal{Q}}_{M}$ is a predictable process that captures value from the space $L_{2}(H)$ of nuclear operators on $H$, that was introduced by Al-Hussein \cite{Al11}, and will be explain in the next section. 

The main aim of this paper is to find an adapted processes $(Y,Z,U,N)$ in a proper space such that BSDE in the Eq. (\ref{eq1}) hold. Then, it allows us to establish the uniqueness of the viscosity solution of a type of non-local variational inequality.
The following is a list of how this paper is organized. Section 2 introduces certain fundamental notations, assumptions, and preliminaries, as well as the a priori estimation of a series of penalized approximations to the equations. In Section 3, we verify the existence and uniqueness of the BSDE solution using the Yosida approximation approach. In Section 4, an example is provided for illustration of the proposed methodology.
\section{Preliminaries}
Al-Hussein \cite{Al11} established the concepts of space and martingales as follows:
Denote that $\mathcal{M}_{[0,T]}^2(H)$ the vector space of the cadlag square-integrable martingales \{$M(t),0\leq t\leq T$\},
and taking values in the space of $H$, moreover $\mathbb{E}[|{M(t)|_H}^2]<\infty$ for each $t\in[0,T]$.
A Hilbert space with respect to the inner product $(M,N)\mapsto\mathbb{E}[\langle M(T),N(T)\rangle_{H}]$
if $\mathbb{P}$-equivalence classes have been established. Let $\mathcal{M}_{[0,T]}^{2,c}(H)$ just be Hilbert subspace containing continuous square integrable martingale in $H$.
These are  $very\ strongly$   $orthogonal$ for $M,N\in\mathcal{M}_{[0,T]}^2(H)$, for all $[0,T]$-valued stopping times $u$, if we can satisfy
 $\mathbb{E}[M(u)\bigotimes N(u)]=\mathbb{E}[M(0)\bigotimes N(0)]$. In particular, if $N(0)=0$, $\mathbb{E}[M(u)\bigotimes N(u)]=0$, then $M$ and $N$ are very strongly orthogonal.

Let $M\in\mathcal{M}_{[0,T]}^2(H)$, and the process $\langle M\rangle$ is the predictable quadratic variation of $M$
as well as ${\mathcal{Q}}_{M}$ be a predicted process that takes values in the set of positive symmetric elements, 
that is linked to a Dol\'{e}ans measure of $M\bigotimes N$.
We define $\langle\langle M\rangle\rangle_{t}=\int_{0}^{t}\mathcal{Q}_{M}(s)\,{\rm d}\langle M\rangle_{s},$ and assume there exists a predictable process $\mathcal{Q}(t,\cdot)$ which is a symmetric positive definite nuclear operator on $H$ and satisfies $\langle\langle M\rangle\rangle_{t}=\int_{0}^{t}\mathcal{Q}(s)\,{\rm d}s.$

Under the space $L^{*}(H;\mathcal{P},M)$ of processes $\Phi$, we first consider $\mathcal{E}(L(H))$ be the space of predictable simple processes, let
 $\Lambda^{2}(H;\mathcal{P},M)$ be the closure of $\mathcal{E}(L(H))$ in $L^{*}(H;\mathcal{P},M)$, after that the space $\Lambda^{2}(H;\mathcal{P},M)$ be one Hilbert subspace of $L^{*}(H;\mathcal{P},M)$.
Additionally, the stochastic integral $\int \Phi\,{\rm d}M $ is defined for an element $\Phi \in\Lambda^{2}(H;\mathcal{P},M)$, which remain with $\mathcal{M}_{[0,T]}^2(H)$ also fulfills the condition

\begin{eqnarray*}
\mathbb{E}\left|{\int_{0}^{T}\Phi(t)\,{\rm d}M(t) }\right| _{H}^2
 =\mathbb{E}\int_{0}^{T}|\Phi(t)\circ \mathcal{Q}_{M}^{\left. 1 \middle / 2 \right.}(t)|_{L_{2}(H)}^2\,{\rm d}\langle M\rangle_{t}.\end{eqnarray*}

Consider the following spaces \cite{Al09}:

\begin{flalign*}
L_{\mathcal{F}} ^2 (0,T;H):= \Bigg\{\phi:[0,T]\times\Omega\rightarrow H, \phi\mbox{ is predictable also satisfies } \mathbb{E}\int_{0}^{T}|\Phi(t)|_{H}^2\,{\rm d}\langle M\rangle_{t}<\infty \Bigg\};
\end{flalign*}

\begin{flalign*}
\mathcal{S}^2(H):= \Bigg\{\phi:[0,T]\times\Omega\rightarrow H, \phi\mbox{ is continuous, adaptable and satisfies }
\mathbb{E}\Bigg[\sup_{0\leq t \leq
T}|\phi(t)|_{H}^2\Bigg]<\infty \Bigg  \} .
\end{flalign*}

As Al-Husse mentioned in reference \cite{Al09}, $\mathcal{S}^2(H)$ is a separable Banach space which is conforming to the norm
\begin{eqnarray*}
\|\phi\|_\mathcal{S}^2(H)=\Bigg(\mathbb{E}\Bigg[\sup_{0\leq t \leq
T}|\phi(t)|_{H}^2\Bigg]\Bigg)^{\left. 1 \middle / 2 \right.}.\end{eqnarray*}

Let $M\in\mathcal{M}_{[0,T]}^2(H)$ be $M(0)=0$ and consider the following assumptions:
 \begin{enumerate}
  \item[(H1)] The function $F:\Omega\times[0,\infty)\times H\times L^{2}(H)\rightarrow H$ fulfills the requirement $\alpha\in\mathbb{R},\beta,\gamma\geq0,$
also let $\eta$ be one $\mathcal{F}_{t}-$progressively measurable process:
  \item[(H2)] \begin{enumerate}
                \item[(i)]  $F(\cdot,\cdot,y,z)$ is  $\mathcal{F}_{t}-$ progressively measurable,
                \item[(ii)]  $y\mapsto F(t,y,z)$ is continuous, $dp\times dt$  a.e. ,
                \item[(iii)]  $\forall y,y^\prime \in H$ and $\forall z,z^\prime \in L_{2}(H)$

    \ \ \ \  $(F(t,y,z)-F(t,y^\prime,z),y-y^\prime) \leq \alpha |y-y^\prime |^2,$

                \item[(iv)] $|F(t,y,z)-F(t,y,z^\prime)|\leq \beta \|z-z^\prime\|$,
                \item[(iv)] $|F(t,y,0)|\leq \eta_{t}+\gamma |y|$,
                 \item[(vi)]  $\mathbb{E}[\int_{0}^{T}|F(t,0,0)|_{H}^2\,{\rm d}{t}]<\infty$.
              \end{enumerate}
  \item[(H3)]\begin{enumerate}
                \item[(i)]  $\varphi$ is just a valid convex function,
                \item[(ii)]  $\varphi(y)\geq \varphi(0)=0$.
              \end{enumerate}
  \item [(H4)]\begin{enumerate}
                \item[(i)]  $\xi \in L^{2}(\Omega,\mathcal{F}_{T},\mathbb{P};H)$,
                \item[(ii)]  $\mathbb{E}[e^{\lambda t}(|\xi|^{2}+|\varphi(\xi)|)]<\infty$,
                \item[(iii)] $\mathbb{E}[\int_{0}^{t}e^{\lambda t}|\eta(s)^{2}|\,{\rm d}{s}]<\infty$,    \ \ \ \ here $\lambda>2\alpha+\beta^2$.
              \end{enumerate}
  \item[(H5)] Every $H$-valued square integrable martingale with filtering $\{\mathcal{F}_{t},0\leq t\leq T\}$ has a continuous version.
\end{enumerate}

We introduce $\varphi$, that is a subdifferential of the  l.s.c. convex function from the space $H$ to $\mathbb{R}$.
$\partial\varphi$ is a multivalued function from the space $H$ to $H$, which was presented by Pardoux and R\u{a}\c{s}canu \cite {Par90}.

For any $u\in H$,
\begin{eqnarray*}
\partial\varphi(u)=\{h\in H:(h,v-u)+\varphi(u)\leq\varphi(v), \forall v\in H\}.\end{eqnarray*}
We let Dom($\partial\varphi$) be the set of $u\in H$ such that $\partial\varphi(u)$ is not empty, and define $(u,v)\in\partial\varphi$ to imply that $u\in$ Dom($\partial\varphi$) and $v\in\partial\varphi(u) $.

The function $\varphi$ is then approximated by the convex $C^1$-function $\varphi_\varepsilon$, $\varepsilon\in (0,1]$ concept was defined by Pardoux and R\u{a}\c{s}canu \cite{Par98} as
\begin{eqnarray*}
\varphi_{\varepsilon}(u)=\inf{\bigg\{\frac{1}{2}|u-v|^{2}+\varepsilon\varphi(v):v\in H \bigg\}}
=\frac{1}{2}|u-J_{\varepsilon}u|^{2}+\varepsilon\varphi(J_{\varepsilon}u) ,
\end{eqnarray*}
where $J_{\varepsilon}u=(I+\varepsilon\partial\varphi)^{-1}(u).$ For all $u,v\in H,\varepsilon>0,$ the properties of approximation presented by Barbu \cite{Bar76} are given by

\begin{eqnarray*}
\begin{array}{ll}
\frac{1}{\varepsilon}D\varphi_{\varepsilon}(u)=\frac{1}{\varepsilon}\partial\varphi_{\varepsilon}(u)
=\frac{1}{\varepsilon}(u-J_{\varepsilon}u)\in \partial\varphi(J_{\varepsilon}u),\\\\
|\varphi(J_{\varepsilon}u)-\varphi(J_{\varepsilon}v)|\leq |u-v|, {\rm and}
\lim\limits_{\varepsilon\to 0}J_{\varepsilon}u= Pr_{\overline{Dom\varphi}}(u).
\end{array}
\end{eqnarray*}

Hence, for all $u,v\in H, \varepsilon>0, \varepsilon^\prime >0,$  we have $0\leq\varphi_\varepsilon\leq (D\varphi_\varepsilon(u),u)$ where
\begin{equation}\label{eq2}
\left(\frac{1}{\varepsilon}D\varphi_{\varepsilon}(u)-\frac{1}{\varepsilon^\prime}D\varphi_{\varepsilon^\prime}(v)\right)\geq -\left(\frac{1}{\varepsilon}+\frac{1}{\varepsilon^\prime}\right)|D\varphi_{\varepsilon}(u)|\times|D\varphi_{\varepsilon^\prime}(v)|.
\end{equation}

Consider the approximating equation
\begin{equation}\label{eq3}
Y_{t}^{\varepsilon}+\frac{1}{\varepsilon}\int_{t}^{T}D\varphi_{\varepsilon}(Y_{s}^{\varepsilon})\,{\rm d}s
= \xi+\int_{t}^{T}F(t,Y_{s}^{\varepsilon},Z_{s}^{\varepsilon}\mathcal{Q}_{s}^{1/2})\,{\rm d}{s}
-\int_{t}^{T}Z_{s}^{\varepsilon}\,{\rm d}{M_s}-\int_{t}^{T}\,{\rm d}{N_{s}^{\varepsilon}}.
\end{equation}
As a result of the conclusion of Al-Hussein \cite{Al09}, for this equation exists a unique solution $(Y^\varepsilon,Z^\varepsilon,N^\varepsilon)\in S_{[0,T]}^{2}(H)\times M_{[0,T]}^{2}(L^{2}(H))\times M_{[0,T]}^{2}(H).$

\begin{lemma}\label{lemma1}
Let assumptions (H1)-(H5) be satisfied, then for all $0\leq a\leq T$,

\begin{equation}\label{eq4}
\mathbb{E}\left[\sup_{a\leq t\leq T}{\rm e}^{\lambda t}|Y_t^\varepsilon|^2
+ \int_{a}^{T}{\rm e}^{\lambda t}(|Y_s^\varepsilon|^2 + \|Z_{s}^{\varepsilon}\mathcal{Q}_{s}^{1/2}\|^2)\,{\rm d}s+\int_{a}^{T}{\rm e}^{\lambda s}\,{\rm d}\langle N\rangle_s\right]\\
\leq C\Gamma_{1}(a,T),
\end{equation}
where
$$\Gamma_{1}(a,T)=\mathbb{E}\Bigg[{\rm e}^{\lambda T}|\xi|^2+\int_{a}^{T}{\rm e}^{\lambda s}|F(s,0,0)|^2\,{\rm d}s\Bigg].$$
\end{lemma}

\begin{proof} Firstly, It\^{o}'s formula for ${\rm e}^{\lambda t}|Y_t^\varepsilon|^2$ yields

\begin{eqnarray*}
{\rm e}^{\lambda t}|Y_t^\varepsilon|^2+\int_{t}^{T}{\rm e}^{\lambda s}(\lambda|Y_s^\varepsilon|^2 + \|Z_{s}^{\varepsilon}\mathcal{Q}_{s}^{1/2}\|^2)\,{\rm d}s+ \int_{t}^{T}{\rm e}^{\lambda s}\,{\rm d}\langle N\rangle_s
+\frac{2}{\varepsilon}\int_{t}^{T}{\rm e}^{\lambda s}(Y_s^\varepsilon,D\varphi_{\varepsilon}(Y_{s}^{\varepsilon})\,{\rm d}s\qquad\quad\\
={\rm e}^{\lambda T}|\xi|^2
+2\int_{t}^{T}{\rm e}^{\lambda s}(Y_s^\varepsilon,F(t,Y_{s}^{\varepsilon},Z_{s}^{\varepsilon}\mathcal{Q}_{s}^{1/2})\,{\rm d}{s}
-2\int_{t}^{T}{\rm e}^{\lambda s}(Y_s^\varepsilon,Z_{s}^{\varepsilon}\,{\rm d}{M_s})-2\int_{t}^{T}{\rm e}^{\lambda s}(Y_s^\varepsilon,\,{\rm d}{N_s^\varepsilon}).\quad
\end{eqnarray*}

Then applying Schwarz's inequalites and considering $\left(\frac{1}{\varepsilon}D\varphi_{\varepsilon}(y),y\right)\geq 0$,
 \begin{eqnarray*}
2(y,F(s,y,z))
&=&2(y,F(s,y,z)-F(s,y,0)+F(s,y,0)-F(s,0,0)+F(s,0,0))\\
&\leq & 2\beta(y,z)+2\alpha|y|^2+2(y,F(s,0,0))\\
&\leq & (2\alpha+(1+r)\beta^2 +r)|y|^2 +\frac{1}{1+r}\|z\|^2 +\frac{1}{r}|F(s,0,0)|^2,
\end{eqnarray*}
we choose $\lambda>2\alpha+\beta^2,0\leq r\leq \frac{\lambda-(2\alpha+\beta^2)}{1+\beta^2}\wedge 1.$
Hence,
\begin{eqnarray*}
{\rm e}^{\lambda t}|Y_t^\varepsilon|^2+\int_{t}^{T}{\rm e}^{\lambda s}[(\lambda-2\alpha-\beta^2-r(1+\beta^2))|Y_s^\varepsilon|^2
+\frac{r}{r+1}\|Z_{s}^{\varepsilon}\mathcal{Q}_{s}^{1/2}\|^2]\,{\rm d}{s}+\int_{t}^{T}{\rm e}^{\lambda s}\,{\rm d}\langle N\rangle_s\\
\leq {\rm e}^{\lambda T}|\xi|^2+\frac{1}{r}\int_{t}^{T}{\rm e}^{\lambda s}|F(s,0,0)|^2\,{\rm d}s
-2\int_{t}^{T}{\rm e}^{\lambda s}(Y_s^\varepsilon,Z_{s}^{\varepsilon}\,{\rm d}{M_s})
-2\int_{t}^{T}{\rm e}^{\lambda s}(Y_s^\varepsilon,\,{\rm d}{N_s}).
\end{eqnarray*}
We obtain
\begin{eqnarray*}
\sup_{a\leq t\leq T}{\rm e}^{\lambda t}|Y_t^\varepsilon|^2
&\leq& {\rm e}^{\lambda T}|\xi|^2+\frac{1}{r}\int_{t}^{T}{\rm e}^{\lambda s}|F(s,0,0)|^2\,{\rm d}s
+2\sup_{a\leq t\leq T}\left|\int_{t}^{T}{\rm e}^{\lambda s}(Y_s^\varepsilon,Z_{s}^{\varepsilon}\,{\rm d}{M_s})\right|\\
&&+2\sup_{a\leq t\leq T}\left|\int_{t}^{T}{\rm e}^{\lambda s}(Y_s^\varepsilon,\,{\rm d}{N_s})\right|.
\end{eqnarray*}
Then, we take the expectation in the above inequality using Burkholder-Davise-Gundy's inequality,
\begin{eqnarray*}
\mathbb{E}\left(\sup_{a\leq t\leq T}{\rm e}^{\lambda t}|Y_t^\varepsilon|^2\right)
&\leq& C_1 +2\mathbb{E}\left(\sup_{a\leq t\leq T}\left|\int_{t}^{T}{\rm e}^{\lambda s}(Y_s^\varepsilon,Z_{s}^{\varepsilon}\,{\rm d}{M_s})\right|\right)\\
&&+2\mathbb{E}\left(\sup_{a\leq t\leq T}\left|\int_{t}^{T}{\rm e}^{\lambda s}(Y_s^\varepsilon,\,{\rm d}{N_s})\right|\right)\\
&\leq& C_1+\frac{1}{2}\mathbb{E}\left(\sup_{a\leq t\leq T}{\rm e}^{\lambda t}|Y_t^\varepsilon|^2\right)
+C_{2}\mathbb{E}\int_{a}^{T}{\rm e}^{\lambda t}\|Z_{s}^{\varepsilon}\mathcal{Q}_{s}^{1/2}\|^2\,{\rm d}{s}\\
&&+\frac{1}{2}\mathbb{E}\left(\sup_{a\leq t\leq T}{\rm e}^{\lambda t}|Y_t^\varepsilon|^2\right)
+C_{3}\mathbb{E}\int_{a}^{T}{\rm e}^{\lambda t}\,{\rm d}\langle N\rangle_s.\\
\end{eqnarray*}
Hence, the proof is completed.
\end{proof}

\begin{lemma}\label{lemma2}
Let assumptions (H1)-(H5) be satisfied, then there exists a positive constant $C$ such that for $0\leq a\leq T$,
\begin{eqnarray}
&&\mathbb{E}\int_{a}^{T}{\rm e}^{\lambda s}\left(\frac{1}{\varepsilon}D\varphi_{\varepsilon}(Y_{s}^{\varepsilon})\right)^{2}\,{\rm d}{s}\leq C\Gamma_{2}(a,T),\label{eq5}\\
&&\mathbb{E}{\rm e}^{\lambda a}\varphi(J_{\varepsilon}Y_{a}^{\varepsilon})
+\mathbb{E}\int_{a}^{T}{\rm e}^{\lambda s}\varphi(J_{\varepsilon}Y_{a}^{\varepsilon})\,{\rm d}{s}\leq C\Gamma_{2}(a,T),\label{eq6}\\
&&\mathbb{E}({\rm e}^{\lambda a}|Y_{a}^{\varepsilon}-J_{\varepsilon}Y_{a}^{\varepsilon}|^2)
\leq\varepsilon^{2}C\Gamma_{2}(a,T),\label{eq7}
\end{eqnarray}

where \begin{eqnarray*}\Gamma_{2}(a,T)=\mathbb{E}\Bigg[{\rm e}^{\lambda T}(|\xi|^{2}+\varphi(\xi))+\int_{a}^{T}{\rm e}^{\lambda s}|\eta(s)|^2\,{\rm d}{s}\Bigg].\end{eqnarray*}
\end{lemma}

\begin{proof}
Consider the subdifferential inequality below
\begin{eqnarray*}
{\rm e}^{\lambda s}\varphi_{\varepsilon}(Y_{s}^{\varepsilon})\geq ({\rm e}^{\lambda s}-{\rm e}^{\lambda r})\varphi_{\varepsilon}(Y_{s}^{\varepsilon})+{\rm e}^{\lambda r}\varphi_{\varepsilon}(Y_{r}^{\varepsilon})
+{\rm e}^{\lambda r}(D\varphi_{\varepsilon}(Y_{r}^{\varepsilon}),Y_{s}^{\varepsilon}-Y_{r}^{\varepsilon})
\end{eqnarray*}
for $s=t_{i+1}\wedge T$, $r=t_{i}\wedge T$, where $t=t_{0}<t_{1}<t_{2}<\cdots$ and $t_{i+1}-t_{i}=1/n$,
summing up over $i$, and passing to the limit as $n\rightarrow\infty$, $\forall t \in [0,T]$, we deduce it to be
\begin{eqnarray*}
&&{\rm e}^{\lambda t}\varphi_{\varepsilon}(Y_{t}^{\varepsilon})+\int_{t}^{T}\lambda{\rm e}^{\lambda s}\varphi_{\varepsilon}(Y_{s}^{\varepsilon})\,{\rm d}{s}
+\frac{1}{\varepsilon}\int_{t}^{T}{\rm e}^{\lambda s}|D\varphi_{\varepsilon}(Y_{s}^{\varepsilon})|^2\,{\rm d}{s}\\
&&\leq {\rm e}^{\lambda T}\varphi_{\varepsilon}(\xi)+\int_{t}^{T}(D\varphi_{\varepsilon}(Y_{s}^{\varepsilon}),
F(s,Y_{s}^{\varepsilon},Z_{s}^{\varepsilon}))\,{\rm d}{s}\\
&&~~~-\int_{t}^{T}{\rm e}^{\lambda s}(D\varphi_{\varepsilon}(Y_{s}^{\varepsilon}),Z_{s}^{\varepsilon}\,{\rm d}{M_s}).
\end{eqnarray*}
As a consequence, we get the result by Eq.(\ref{eq4}), as also as the Proposition 2.2 from Pardoux and Răşcanu \cite{Par98}. 
\begin{eqnarray*}
&&\frac{1}{2}|D\varphi_{\varepsilon}(y)|^{2}+\varepsilon\varphi(J_{\varepsilon}y)=\varphi_{\varepsilon}(y),
~~\varepsilon\varphi(J_{\varepsilon}y)\leq \varphi_{\varepsilon}(y),~~~~~~~~~~~~~~~~~~~~~~~~~~~~\\
&&-\lambda\varphi_{\varepsilon}(y)\leq |\lambda|\varphi_{\varepsilon}(y)\leq |\lambda|(D\varphi_{\varepsilon}(y),y),\\
&&\varphi_{\varepsilon}(\xi)\leq \varepsilon\varphi(\xi),\\
&& 0\le\varphi_{\varepsilon}(u)\le (D\varphi_{\varepsilon}(u),u),
\end{eqnarray*}

\begin{eqnarray*}
(D\varphi_{\varepsilon}(y),|\lambda|y+F(s,y,z))
&\leq&\frac{1}{2\varepsilon}|D\varphi_{\varepsilon}(y)|^{2}
+\frac{\varepsilon}{2}(|\lambda||y|+|F(s,y,z)|)^{2}\\
&\leq& \frac{1}{2\varepsilon}|D\varphi_{\varepsilon}(y)|^{2}+\varepsilon(|\lambda|^{2}|y|^{2}+|F(s,y,z)|^{2})\\
&\leq& \frac{1}{2\varepsilon}|D\varphi_{\varepsilon}(y)|^{2}+\varepsilon[|\lambda|^{2}|y|^{2}\\
&&+4(\beta^{2}\|z\|^{2}+\gamma^{2}|y|^{2}+\eta^{2}(s))].
\end{eqnarray*}
And the end outcome is obtained.
\end{proof}

\begin{lemma}\label{lemma3}
Assume that assumptions (H1)-(H5) are satisfied, then for any $\varepsilon,\varepsilon^\prime >0$,
\begin{eqnarray}
&&\mathbb{E}\Bigg[\int_{0}^{T}{\rm e}^{\lambda s}(|Y_s^\varepsilon-Y_{s}^{\varepsilon^\prime}|^{2}+
\|Z_{s}^{\varepsilon}\mathcal{Q}_{s}^{1/2}-Z_{s}^{\varepsilon^\prime}\mathcal{Q}_{s}^{1/2}\|^2)\,{\rm d}{s}\Bigg]
\leq (\varepsilon+\varepsilon^\prime)C\Gamma(T),\label{eq8}\\
&&\mathbb{E}\left(\sup_{0\leq t\leq T}{\rm e}^{\lambda t}|Y_s^\varepsilon-Y_{s}^{\varepsilon^\prime}|^{2}\right)
\leq (\varepsilon+\varepsilon^\prime)C\Gamma(T),\label{eq9}
\end{eqnarray}
where \begin{eqnarray*}\Gamma(T)=\mathbb{E}\Bigg[{\rm e}^{\lambda T}(|\xi|^{2}+\varphi(\xi))+\int_{0}^{T}|F(s,0,0)|^2\,{\rm d}{s}\Bigg].\end{eqnarray*}.
\end{lemma}
\begin{proof}  Firstly, It\^{o}'s formula for ${\rm e}^{\lambda t}|Y_s^\varepsilon-Y_{s}^{\varepsilon^\prime}|^2$ yields
\begin{eqnarray}\label{eq10}
\begin{split}
&&{\rm e}^{\lambda t}|Y_t^\varepsilon-Y_{s}^{\varepsilon^\prime}|^2+\int_{t}^{T}{\rm e}^{\lambda s}(\lambda|Y_s^\varepsilon-Y_{s}^{\varepsilon^\prime}|^2 + \|Z_{s}^{\varepsilon}\mathcal{Q}_{s}^{1/2}-Z_{s}^{\varepsilon^\prime}\mathcal{Q}_{s}^{1/2}\|^2)\,{\rm d}s\\
&&+\int_{t}^{T}{\rm e}^{\lambda s}\,{\rm d}\langle N^{\varepsilon}-N^{\varepsilon^\prime}\rangle_s
+2\int_{t}^{T}{\rm e}^{\lambda s}(Y_s^\varepsilon-Y_{s}^{\varepsilon^\prime},
\frac{1}{\varepsilon}D\varphi_{\varepsilon}(Y_s^\varepsilon)
-\frac{1}{\varepsilon^\prime}D\varphi_{\varepsilon^\prime}(Y_{s}^{\varepsilon^\prime}))\,{\rm d}s\\
&&={\rm e}^{\lambda T}|\xi|^2
+2\int_{t}^{T}{\rm e}^{\lambda s}(Y_s^\varepsilon-Y_{s}^{\varepsilon^\prime},F(t,Y_{s}^{\varepsilon},Z_{s}^{\varepsilon}\mathcal{Q}_{s}^{1/2})-
F(t,Y_{s}^{\varepsilon},Z_{s}^{\varepsilon}\mathcal{Q}_{s}^{1/2}))\,{\rm d}{s}\\
&&~~~-2\int_{t}^{T}{\rm e}^{\lambda s}(Y_s^\varepsilon-Y_{s}^{\varepsilon^\prime},(Z_{s}^{\varepsilon}-Z_{s}^{\varepsilon^\prime})\,{\rm d}{M_s})
-2\int_{t}^{T}{\rm e}^{\lambda s}(Y_s^\varepsilon-Y_{s}^{\varepsilon^\prime},{\rm d}{N_s^\varepsilon}-{\rm d}{N_{s}^{\varepsilon^\prime}}).
\end{split}
\end{eqnarray}

Moreover,
\begin{eqnarray*}
&&2(Y_s^\varepsilon-Y_{s}^{\varepsilon^\prime},F(t,Y_{s}^{\varepsilon},Z_{s}^{\varepsilon}\mathcal{Q}_{s}^{1/2})-
F(t,Y_{s}^{\varepsilon^\prime},Z_{s}^{\varepsilon^\prime}\mathcal{Q}_{s}^{1/2}))
\\
&&\qquad
\leq (2\alpha+(1+r)\beta^{2})|Y_t^\varepsilon-Y_{s}^{\varepsilon^\prime}|^2 +\frac{1}{1+r}\|Z_{s}^{\varepsilon}\mathcal{Q}_{s}^{1/2}-Z_{s}^{\varepsilon^\prime}\mathcal{Q}_{s}^{1/2}\|^2
\end{eqnarray*}
 by Eq.(\ref{eq2}), hence,
\begin{eqnarray*}
&&{\rm e}^{\lambda t}|Y_t^\varepsilon-Y_{s}^{\varepsilon^\prime}|^2+\int_{t}^{T}{\rm e}^{\lambda s}\,{\rm d}\langle N^{\varepsilon}-N^{\varepsilon^\prime}\rangle_s\\
&&~~~+\int_{t}^{T}{\rm e}^{\lambda s}\Bigg[(\lambda -2\alpha-\beta^2 -r\beta^2)|Y_s^\varepsilon-Y_{s}^{\varepsilon^\prime}|^2 + \frac{r}{r+1}\|Z_{s}^{\varepsilon}\mathcal{Q}_{s}^{1/2}-Z_{s}^{\varepsilon^\prime}\mathcal{Q}_{s}^{1/2}\|^2\bigg]\,{\rm d}s\\
&&\leq 2\left(\frac{1}{\varepsilon}+\frac{1}{\varepsilon^\prime}\right)\int_{t}^{T}{\rm e}^{\lambda s}
|D\varphi_{\varepsilon}Y_t^\varepsilon|\times|D\varphi_{\varepsilon^\prime}Y_{s}^{\varepsilon^\prime}|\,{\rm d}s
-2\int_{t}^{T}{\rm e}^{\lambda s}(Y_s^\varepsilon-Y_{s}^{\varepsilon^\prime},Z_{s}^{\varepsilon}-Z_{s}^{\varepsilon^\prime}\,{\rm d}{M_s})\\
&&~~~-2\int_{t}^{T}{\rm e}^{\lambda s}(Y_s^\varepsilon-Y_{s}^{\varepsilon^\prime},{\rm d}{N_s^\varepsilon}-{\rm d}{N_{s}^{\varepsilon^\prime}}).
\end{eqnarray*}
Take expectations on both sides of the above inequation, and combine the inequation bellow from Lemma \ref{lemma2},
\begin{eqnarray*}
2\left(\frac{1}{\varepsilon}+\frac{1}{\varepsilon^\prime}\right) \mathbb{E}\int_{t}^{T}{\rm e}^{\lambda s}
|D\varphi_{\varepsilon}Y_t^\varepsilon|\times|D\varphi_{\varepsilon^\prime}Y_{s}^{\varepsilon^\prime}|\,{\rm d}s
\leq C(\varepsilon+\varepsilon^\prime)\Gamma(T).
\end{eqnarray*}
We can obtain
\begin{eqnarray*}
&&\mathbb{E}\Bigg[\int_{0}^{T}{\rm e}^{\lambda s}(|Y_s^\varepsilon-Y_{s}^{\varepsilon^\prime}|^{2}){\rm d}{s}\Bigg]
\leq C(\varepsilon+\varepsilon^\prime)\Gamma(T),
\end{eqnarray*}
On the other hand, on the basis of Eq.(\ref{eq10}), we obtain
\begin{eqnarray*}
&&\mathbb{E}\Bigg[{\rm e}^{\lambda s}\|Z_{s}^{\varepsilon}\mathcal{Q}_{s}^{1/2}-Z_{s}^{\varepsilon^\prime}\mathcal{Q}_{s}^{1/2}\|^2\,{\rm d}{s}\Bigg]
\leq C(\varepsilon+\varepsilon^\prime)\mathbb{E}\Bigg[\int_{t}^{T}{\rm e}^{\lambda s}|D\varphi_{\varepsilon}Y_t^\varepsilon|\times|D\varphi_{\varepsilon^\prime}Y_{s}^{\varepsilon^\prime}|\,{\rm d}s\Bigg],
\end{eqnarray*}
Indeed, it follows from Burkholder-Davis-Gundy's inequality the result can be obtained
\begin{eqnarray*}
&&\mathbb{E}\left(\sup_{0\leq t\leq T}{\rm e}^{\lambda t}|Y_s^\varepsilon-Y_{s}^{\varepsilon^\prime}|^{2}\right)
\leq (\varepsilon+\varepsilon^\prime)C\Gamma(T),
\end{eqnarray*}
We then complete the proof.
\end{proof}

\section{The existence and uniqueness of the solution}
\begin{lemma} \label{lemma4}
Let the assumptions (H1)-(H5) be satisfied, if $(Y,Z,U,N)$ is a solution for Eq.(\ref{eq1}), and $(Y^\prime,Z^\prime,U^\prime,N^\prime)$ as an alternative solution. Denote by $(\delta Y,\delta Z,\delta U,\delta N) \triangleq  (Y-Y^\prime,Z-Z^\prime,U-U^\prime,N-N^\prime)$, let $\lambda$ be a real number, then we have
\begin{eqnarray}
&&\mathbb{E}\int_{t}^{T}{\rm e}^{\lambda s}(|\delta Y_s|^2+\|\delta Z_{s}\mathcal{Q}_{s}^{1/2}\|^2)\,{\rm d}s=0,\label{eq11}\\
&&\mathbb{E}\left(\sup_{0\leq t\leq T}{\rm e}^{\lambda t}|\delta Y_t|^2\right)=0.\label{eq12}
\end{eqnarray}
\end{lemma}

\begin{proof}   It\^{o}'s formula for ${\rm e}^{\lambda t}|\delta Y_t|^2$ yields
\begin{eqnarray*}
{\rm e}^{\lambda T}|\delta Y_T|^2-{\rm e}^{\lambda t}|\delta Y_t|^2
&=&\int_{t}^{T}{\rm e}^{\lambda s}(\lambda|\delta Y_s|^2 + \|\delta Z_{s}\mathcal{Q}_{s}^{1/2}\|^2)\,{\rm d}s+ \int_{t}^{T}{\rm e}^{\lambda s}\,{\rm d}\langle\delta N\rangle_s\\
&&+2\int_{t}^{T}{\rm e}^{\lambda s}(\delta Y_s,[F(s,Y_s,Z_{s}\mathcal{Q}_{s}^{1/2})-F(s,\delta Y_s,\delta Z_{s}\mathcal{Q}_{s}^{1/2})]\,{\rm d}s\\
&&~~~~~~~~~~~~~-|\delta U_s|\,{\rm d}s-|\delta Z_s|\,{\rm d}{M_s}-\,{\rm d}{\delta N_s}).
\end{eqnarray*}
Taking the expectation,
\begin{eqnarray*}
&&\mathbb{E}[{\rm e}^{\lambda t}|\delta Y_t|^2]+\mathbb{E}\int_{t}^{T}{\rm e}^{\lambda s}(\lambda|\delta Y_s|^2 + \|\delta Z_{s}\mathcal{Q}_{s}^{1/2}\|^2)\,{\rm d}s+\mathbb{E}\int_{t}^{T}{\rm e}^{\lambda s}\,{\rm d}\langle\delta N\rangle_s\\
&&~~~+2\mathbb{E}\int_{t}^{T}{\rm e}^{\lambda s}(\delta Y_t,\delta U_s)\,{\rm d}s\\
&&=2\mathbb{E}\int_{t}^{T}{\rm e}^{\lambda s}(\delta Y_t,[F(s,Y_s,Z_{s}\mathcal{Q}_{s}^{1/2})-F(s,\delta Y_s,\delta Z_{s}\mathcal{Q}_{s}^{1/2})])\,{\rm d}s.
\end{eqnarray*}
However,\\
$2(\delta Y_t,\delta U_s)\geq 0,$
\begin{eqnarray*}
(\delta Y_t,F(s,Y_s,Z_{s}\mathcal{Q}_{s}^{1/2})-F(s,\delta Y_s,\delta Z_{s}\mathcal{Q}_{s}^{1/2}))
&\leq&2\alpha|\delta Y_t|^2+\beta(|\delta Y_s|^2 + \|\delta Z_{s}\mathcal{Q}_{s}^{1/2}\|^2)\\
&\leq&(2\alpha+\beta^2 +r\beta^2)|\delta Y_t|^2
+\frac{1}{1+r}\|\delta Z_{s}\mathcal{Q}_{s}^{1/2}\|^2.
\end{eqnarray*}
Then, we obtain Eq.(\ref{eq11}). Finally we obtain Eq.(\ref{eq12}) by using the Burkholder-Davis-Gundy's inequality.
\end{proof}
\begin{theorem}\label{thm1}
Suppose that the conditions (H1)-(H5) hold, and then there will exist a unique quadruple  $(Y,Z,U,N)$ so that\\
$$\begin{array}{ll} 
 Y\in S_{[0,T]}^{2}(H),\ Z\in M_{[0,T]}^{2}(L^{2}(H)),\ U\in M_{[0,T]}^{2}(H),\ N\in M_{[0,T]}^{2}(H), \ \\
\end{array}$$
\begin{eqnarray}
&&\mathbb{E}\int_{t}^{T}{\rm e}^{\lambda s}\varphi(Y_s)\,{\rm d}s\leq\infty,\label{eq13}\\
&&\left\{
 \begin{array}{ll}
 \displaystyle Y_{t}+\int_{t}^{T}U_s\,{\rm d}s
= \xi+\int_{t}^{T}F(t,Y_{s},Z_{s}\mathcal{Q}_{s}^{1/2})\,{\rm d}{s}
-\int_{t}^{T}Z_{s}\,{\rm d}{M_s}-\int_{t}^{T}\,{\rm d}{N_s},\\
\\
 \displaystyle Y_t\in \,{\rm Dom}(\partial\varphi), \,{\rm and\ }   U_t\in \partial\varphi(Y_t).
\end{array}\right.\label{eq14}
  \end{eqnarray}
\end{theorem}
\begin{proof}  Uniqueness can be proven by using Lemma \ref{lemma4} above. As a limit of the quadruple 
$(Y_s^\varepsilon,Z_s^\varepsilon,\frac{1}{\varepsilon}D\varphi_\varepsilon (Y_s^\varepsilon),N_s^\varepsilon)$, the existence of the solution $(Y,Z,U,N)$ is established.
The following result comes from Lemma \ref{lemma3}.
\begin{eqnarray*}
\begin{array}{ll}
\exists Y\in S_{[0,T]}^{2}(H), Z\in M_{[0,T]}^{2}(L^{2}(H)),\\\\
\lim\limits_{\varepsilon\to 0}Y^\varepsilon=Y \,{\rm in}\ S_{[0,T]}^{2}(H),\\\\
\lim\limits_{\varepsilon\to 0}Z^\varepsilon=Z \,{\rm in}\ M_{[0,T]}^{2}(L^{2}(H)),
\end{array}\end{eqnarray*}
and from Eqs.(\ref{eq5}) and (\ref{eq7}), for $\forall t\in [0,T]$, we have
\begin{eqnarray*}
\begin{array}{ll}
\lim\limits_{\varepsilon\to 0}J_{\varepsilon}(Y^\varepsilon)=Y \,{\rm in}\ S_{[0,T]}^{2}(H),\\\\
\lim\limits_{\varepsilon\to 0}\mathbb{E}({\rm e}^{\lambda t}|J_{\varepsilon}(Y_t^\varepsilon)-Y_t|^2)=0.
\end{array}\end{eqnarray*}
Eq.(\ref{eq13}) follow from Eq.(\ref{eq6}) and Eq.(\ref{eq11}).

For all $\varepsilon>0$, we let $U_t^\varepsilon=\frac{1}{\varepsilon}D\varphi_\varepsilon(Y_t^\varepsilon),$
as well as $\widehat{U}_t^\varepsilon=\int_{t}^{T}U_s^\varepsilon\,{\rm d}s. $
As a result of the convergence result which was presented by Pardoux and Rascanu \cite{Par98} and Eq.(\ref{eq13}), there exists a progressively measurable process $\{\widehat{U}_t,0\leq t\leq T\}$ so that for each $T>0$,
\begin{eqnarray*}
\mathbb{E}\left(\sup_{0\leq t\leq T}|\widehat{U}_t^\varepsilon-\widehat{U}_t|^2\right)\rightarrow 0,  \ \varepsilon\rightarrow 0.
\end{eqnarray*}
Furthermore, from Eq.(\ref{eq5}),
\begin{eqnarray*}
\sup_{\varepsilon>0}\mathbb{E}\int_{0}^{T}{\rm e}^{\lambda t}|U_t^\varepsilon|^2\,{\rm d}s<\infty.
\end{eqnarray*}

In the space $L^{2}(\Omega,H^{1}(0,T))$, $\widehat{U}^\varepsilon$ is bounded for all $T<0$, as well as $\lim\limits_{\varepsilon\to 0}\widehat{U}^\varepsilon=\widehat{U}$ in $L^{2}(\Omega,H^{1}(0,T))$, specifically, $\widehat{U}_t$ adopts the form $\widehat{U}_t=\int_{t}^{T}U_s\,{\rm d}s$, where $\{\widehat{U}_t,0\leq t\leq T\}$ is gradually measurable, and $\widehat{U}$ is completely continuous.

As well as from  Gegout-Petit and Pardoux's Lemma 5.8 \cite{Geg96}, for each $0\leq a\leq b\leq T$, $V\in M_{[a,b]}^{2}(H),$
\begin{eqnarray*}
\int_{a}^{b}(U_t^\varepsilon,V_t -Y_t^\varepsilon)\,{\rm d}t \rightarrow \int_{a}^{b}(U_t,V_t -Y_t)\,{\rm d}t
\end{eqnarray*}
in probability, and from Eq.(\ref{eq5}) $\int_{a}^{b}(U_t^\varepsilon,J_{\varepsilon}(Y_t^\varepsilon) -Y_t^\varepsilon)\,{\rm d}t \rightarrow 0.$ Now, since $U_t^\varepsilon\in \partial\varphi(J_{\varepsilon}(Y_t^\varepsilon))$,
\begin{eqnarray*}
\int_{a}^{b}(U_t^\varepsilon,V_t -J_{\varepsilon}(Y_t^\varepsilon))\,{\rm d}t
+ \int_{a}^{b}\varphi(J_{\varepsilon}(Y_t^\varepsilon))\,{\rm d}t
\leq \int_{a}^{b}\varphi(V_t)\,{\rm d}t.
\end{eqnarray*}
Taking the limit inferior in probability from the previous, we get 
\begin{eqnarray*}
\int_{a}^{b}(U_t,V_t -Y_t)\,{\rm d}t
+ \int_{a}^{b}\varphi(Y_t)\,{\rm d}t
\leq \int_{a}^{b}\varphi(V_t)\,{\rm d}t.
\end{eqnarray*}
While the constant $a,b$ and the process $V$ are random, Eq.(\ref{eq14}) has been proven.
\end{proof}

\section{Examples}
Consider Theorem 4.2 and Example 4.3 of Al-Hussein \cite{Al09}, the follow backward stochastic partial differential equation (BSPDE) has its solution $(Y,Z,N)$:
\begin{equation}\label{eq15}
\left\{
\begin{array}{ll}
  \displaystyle \,-{\rm d}Y_{t}=A\,{\rm d}{t}-Z_{t}\,{\rm d}{M_{t}}-\,{\rm d}{N_{t}},\ 0\leq t\leq T,\\
\\
 \displaystyle Y(T)=\xi.
\end{array} \right.
  \end{equation}
Here, we let $F(t,Y_{t},Z_{t}\mathcal{Q}_{t}^{1/2})=0$, and assume $A$ is a linear operator with no bounds from $\mathcal{D}(A)$ to $H$. If $A:V\to {V}^{'}$ is a continuous linear operator, Eq.(\ref{eq14}) has an unique solution.

Now, let $(\Omega, \mathcal{F}, \mathcal{F}_t, P)$ be a full probability space and $\mathcal{D} \subset \mathbb{R}^{d}$ be an open bounded subset with suitably smooth border $\partial(\mathcal{D})$. Let $M_t$ be martingales and set $H=H_1:=L^{2}(\mathcal{D})$.
And then, consider the BSPDE
\begin{equation}\label{eq16}
\left\{
\begin{array}{ll}
  \displaystyle \,-{\rm d}Y_{t}+ \partial j(Y_{t}){\rm d}{t}\ni A\,{\rm d}{t}+F(t,Y_{t},Z_{t}\mathcal{Q}_{t}^{1/2}){\rm d}{t}-Z_{t}\,{\rm d}{M_{t}}-\,{\rm d}{N_{t}},\ 0\leq t\leq T,\\
\\
 \displaystyle Y_t=0   \ \   \mbox{on}  \ \      \Omega\times[0,T]\times\partial(\mathcal{D}).
\end{array} \right.
  \end{equation}

Firstly, let us apply Theorem 3.2 of Maticiuc and R\u{a}\c{s}canu \cite{Mat15}, where $\varphi$ is a function from $L^{2}(\mathcal{D})\to\mathbb{R}$ which is provided bellow:
\begin{equation}
\varphi(u)=\left\{
\begin{array}{ll}
  \displaystyle \frac{1}{2}\int_{\mathcal{D}}|\nabla u(x)|^{2}\,{\rm d}{x}+\int_{\mathcal{D}} j(u(x))\,{\rm d}{x},\ \ \mbox{if}\ \  u\in H_{0}^{1}(\mathcal{D}),j(u)\in L^{1}(\mathcal{D})
\\
 \displaystyle +\infty,   \ \   \mbox{otherwise}. \ \      
\end{array} \right.
  \end{equation}

 Then, consider the Proposition 2.8 of \cite{Bar10}, the following properties hold:
 \begin{enumerate}
\item[(i)] $\varphi$ is a fuction what proper, convex as well as 1.s.c.,
\item[(ii)] $\partial\varphi(u)=\{  u^{\ast}\in L^{2}(\mathcal{D}):\ast\in \partial j(u(x))-\bigtriangleup u(x)\ \ a.e.\  \mbox{on}\ \  \mathcal{D}\}$, $\forall u\in \mbox{Dom}(\partial\varphi)$
\item[(iii)]$\mbox{Dom}(\partial\varphi)=\{u\in H^{1}(\mathcal{D})\bigcap H^{2}(\mathcal{D}): u(x)\in \mbox{Dom}(\partial j)\ \ a.e.\ \mbox{on} \ \mathcal{D}\}$,
\item[(iv)]$ \|u\|\leq C\|u^{*}\|,  \forall (u,u^{*})\in \partial\varphi$.
\end{enumerate}

Lastly, by applying Theorem \ref{thm1} from section 3, we decide that, under the above conditions, Eq.(\ref{eq16}) has an unique solution $(Y,Z,U,N)\in S^{2} (H)\times  M^{2}(L^{2}(H))\times  M^{2} (H)\times  M^{2} (H),\
$ such that $(Y_t,Z_t)=(\eta,0)$, where $\eta$ is a $H^1(\mathcal{D})$-valued random variable, $\mathcal{F}_t$-measurable and  
\begin{enumerate}
\item[(a)] $Y_{t}+\int_{t}^{T}U_s\,{\rm d}s
= \xi+\int_{t}^{T} AY_{s}\,{\rm d}{\mathcal{Q}_{s}}
-\int_{t}^{T}Z_{s}\,{\rm d}{M_s}-\int_{t}^{T}\,{\rm d}{N_s}\ \ a.s.$,

\item[(b)] $Y_{t}\in H^1(\mathcal{D}) \cap H^2(\mathcal{D})$,
\item[(c)] $Y_{t}(x)\in \mbox{Dom}(j)$,
\item[(d)] $U_{t}(x)\in \partial j(Y_{t}(x))$.
\end{enumerate}

\section{Conclusions}
The goal of this paper is to present and study a type of BSDEs, which is driven by infinite-dimensional martingales with subdifferential operators. We have shown that the adaptive solution of this BSDE exists and is unique. Additionally, we have presented a special example for the simple case. For future work we will focus on this interesting problem and pay more attention to the simulation of numerical solutions of BSDEs of multidimensional and even infinite-dimensional types, and their applications in finance and computing, such as \cite{Yu19, E17, Tak22}.

\section{Data Availability}
The data used to support the findings of this study are freely available.
\section{Conflicts of Interest}
The authors declare that they have no conflicts of interest.
\section{Acknowledgments}
The research was funded by Key scientific research
projects of Suzhou University (2019yzd08), Anhui Philosophy and Social Science Planning Project (AHSKQ2021D98), and University of Malaya research project (BKS073-2017).






\begin{table}[!htbp]
\end{table}

\begin{figure}[!htbp]
\end{figure}



  

\end{document}